%% file: hyperplane-Sn-orbit.tex
\documentclass[11pt,reqno]{amsart}
\usepackage{amsmath,amsthm,amsfonts,amssymb,amscd}
\usepackage{enumerate}
\usepackage[vcentermath,enableskew]{youngtab}
\usepackage{microtype}
\usepackage{tikz}
\usepackage{mathrsfs}
\usepackage[ocgcolorlinks, linkcolor=red!90!black, citecolor=blue!80!black]{hyperref}
\usepackage{mathtools}
\usepackage{accents}

\usetikzlibrary{positioning}

\newtheorem{thm}{Theorem}[section]
\newtheorem*{thm*}{Theorem}
\newtheorem{lemma}[thm]{Lemma}
\theoremstyle{definition} \newtheorem{definition}[thm]{Definition}
\newtheorem*{definition*}{Definition}
\theoremstyle{definition} \newtheorem{example}[thm]{Example}
\newtheorem*{lemma*}{Lemma}

\newtheorem*{corollary*}{Corollary}
\theoremstyle{definition} \newtheorem{remark}[thm]{Remark}

\newtheorem*{conjecture*}{Conjecture}

\input{preamble.tex}

\begin{document}

\author{Brendan Pawlowski}
\title{The fraction of an $S_n$-orbit on a hyperplane}

\begin{abstract}
    Huang, McKinnon, and Satriano conjectured that if $v \in \RR^n$ has distinct coordinates and $n \geq 3$, then a hyperplane through the origin other than $\sum_i x_i = 0$ contains at most $2\lfloor n/2 \rfloor (n-2)!$ of the vectors obtained by permuting the coordinates of $v$. We prove this conjecture.
\end{abstract}  

\maketitle

\section{Introduction}

A permutation $\sigma \in S_n$ acts on $v \in \RR^n$ by $\sigma v = (v_{\sigma^{-1}(1)}, \ldots, v_{\sigma^{-1}(n)})$. Define $\pair(v,w) = \#\{\sigma \in S_n : w \cdot \sigma v = 0\}$, where $\cdot$ is the usual dot product on $\RR^n$. For example, if $\one = (1,1,\ldots,1) \in \RR^n$ is the all-ones vector, then $\pair(v,1)$ is $n!$ or $0$ depending on whether $v \cdot \one$ is zero or not. Putting this degenerate case aside, the goal of this paper is to prove the following theorem.
\begin{thm} \label{thm:main} For $n \geq 3$, the maximum of $\pair(v,w)$ over all $v, w \in \RR^n$ such that $v$ has distinct coordinates and $v \cdot \one \neq 0$ is 
    \begin{equation*}
        \begin{cases} 
            (n-1)! & \text{for $n$ odd}\\
            n(n-2)! & \text{for $n$ even}
         \end{cases} = 2 \lfloor \tfrac{n}{2} \rfloor (n-2)!,
        \end{equation*}
\end{thm}
This was conjectured by Huang, McKinnon, and Satriano, who proved the conjecture in some cases and gave explicit vectors achieving the conjectured bound \cite{Sn-orbit-hyperplane}. The problem is therefore to show that $\pair(v,w) \leq 2\lfloor n/2 \rfloor (n-2)!$ for all appropriate $v$ and $w$. (We do not know what happens if $v$ is allowed to have repeated coordinates.)

Here is an outline of our argument. We observe that if $v$ has distinct coordinates, then $\{\sigma \in S_n : w \cdot \sigma v = 0\}$ is an antichain in a certain weakening of Bruhat order on $S_n$, isomorphic to a disjoint union of copies of Bruhat order on $S_n/S_{\alpha}$ for some parabolic subgroup $S_{\alpha}$. The latter poset is known to have the Sperner property, so that $\#\{\sigma \in S_n : w \cdot \sigma v = 0\}$ is bounded in terms of the largest rank in $S_n/S_{\alpha}$. These ranks are the coefficients of the $q$-multinomial coefficient $\smallqbinom{n}{\alpha}$, and Theorem~\ref{thm:main} will follow from an appropriate bound on those coefficients. We note that this argument is similar in outline to arguments of Stanley \cite[\S 4.1.3]{stanley-algebra-to-combinatorics} and of Lindstr\"om \cite{lindstrom-sperner} resolving problems in extremal combinatorics using the Sperner property. 

\section{Proof of Theorem~\ref{thm:main}}

A \emph{composition} of a nonnegative integer $n$ is a sequence $\alpha = (\alpha_1, \ldots, \alpha_m)$ of positive integers with $|\alpha| := \sum_i \alpha_i = n$. Write $\alpha \vDash n$ to indicate that $\alpha$ is a composition of $n$, and let $\ell(\alpha)$ be the length of $\alpha$. Let $\set(\alpha)$ denote the set $\{\alpha_1 + \cdots + \alpha_i : 1 \leq i < \ell(\alpha)\}$. For example, $\set((1,3,1,2)) = \{1,4,5\}$. 

An \emph{ordered set partition of type $\alpha$} is a sequence $B_\bullet = (B_1, \ldots, B_{\ell(\alpha)})$ of nonempty sets with $|B_i| = \alpha_i$ whose disjoint union is $\{1,2,\ldots,|\alpha|\}$. For example, $(\{5\},\{1,4,6\},\{3\},\{2,7\})$ is an ordered set partition of type $(1,3,1,2)$. Let $S_n$ act on subsets of $\{1,2,\ldots,n\}$ in the obvious way, and hence on ordered set partitions of type $\alpha$.

If $\set(\alpha) = \{i_1 < \cdots < i_{m-1}\}$ and we set $i_0 = 0$ and $i_m = n$, the sequence of sets $\{i_j+1, i_j+2, \ldots, i_{j+1}\}$ for $j = 0,1,\ldots,m-1$ is an ordered set partition of type $\alpha$. Let $S_{\alpha}$ be the subgroup of $S_n$ preserving each of these sets. For example, $S_{(1,3,1,2)}$ is the subgroup of permutations in $S_7$ mapping each set in $(\{1\},\{2,3,4\},\{5\},\{6,7\})$ to itself. We abbreviate the order of $S_{\alpha}$ as $\alpha! := \prod_i \alpha_i!$. The $S_n$-action on ordered set partitions of type $\alpha$ is transitive and $S_{\alpha}$ is the stabilizer of a point, so we may identify $S_n/S_{\alpha}$ with the set of all ordered set partitions of type $\alpha$.

Let $t_{ij} \in S_n$ interchange $i$ and $j$ and fix all other elements of $\{1,2,\ldots,n\}$.
\begin{definition} The \emph{Bruhat order} on $S_n/S_{\alpha}$ is the transitive closure of the relations $(B_1, \ldots, B_m) < t_{ij}(B_1, \ldots, B_m)$ where $i \in B_a$ and $j \in B_b$ with $a < b$ and $i < j$. \end{definition}
    In other words, $B_\bullet \leq B_\bullet'$ if $B_\bullet$ is connected to $B_\bullet'$ by a sequence of relations as in the definition.

\begin{example} \label{ex:S3-mod-S21}
    The Bruhat order on $S_3/S_{(2,1)}$ is the chain
    \begin{equation*}
        \{\{1,2\},\{3\}\} < \{\{1,3\},\{2\}\} < \{\{2,3\},\{1\}\}.
    \end{equation*}
\end{example}

\begin{example}
     Bruhat order on $S_n/S_{(1,1,\ldots,1)}$ is just the usual Bruhat order on $S_n$.
\end{example} 

\begin{definition}
    Given a totally ordered set $B = \{i_1 < \cdots < i_m\}$, let $\word(B)$ be the word $i_1 \cdots i_m$. If $B_\bullet = (B_1, \ldots, B_m)$ is a sequence of such sets, let $\word(B_\bullet)$ be the concatenation $\word(B_1)\cdots \word(B_m)$.
\end{definition}
For example, $\word(\{5\},\{1,4,6\},\{3\},\{2,7\}) = 5146327$. Identifying ordered set partitions of type $\alpha$ with cosets in $S_n/S_{\alpha}$, the map $B_\bullet \mapsto \word(B_\bullet)$ chooses a distinguished representative from each coset.

\begin{remark}  The map $B_\bullet \mapsto \word(B_\bullet)$ sends ordered set partitions of type $\alpha$ to permutations $\sigma$  for which $\sigma_i > \sigma_{i+1}$ implies $i \in \set(\alpha)$, recovering a more common definition of $S_n/S_{\alpha}$:  Bruhat order on $S_n$ restricted to the set of such permutations. \end{remark}

\begin{definition} Define a poset structure on $(S_n/S_{\alpha}) \times S_{\alpha}$ by the relation $(B_\bullet,\sigma) \leq (B'_\bullet,\sigma')$ if and only if $B_\bullet \leq B'_\bullet$ in Bruhat order. The \emph{$\alpha$-Bruhat order}  $<_\alpha$ on $S_n$ is the image of the partial order $(S_n/S_{\alpha}) \times S_{\alpha}$ under the bijection $(S_n/S_{\alpha}) \times S_{\alpha} \to S_n$, $(B_\bullet, \sigma) \mapsto \word(B_\bullet)\sigma$. \end{definition}
Note that $\alpha$-Bruhat order is isomorphic to the disjoint union of $\alpha!$ copies of $S_n/S_{\alpha}$.

\begin{example}
    Comparing with Example~\ref{ex:S3-mod-S21}, the $(2,1)$-Bruhat order on $S_3$ is the disjoint union of the two chains 
    \begin{equation*}
        123 <_{(2,1)} 132 <_{(2,1)} 231 \qquad \text{and} \qquad 213 <_{(2,1)} 312 <_{(2,1)} 321.
    \end{equation*}
\end{example}

We now associate a composition of $n$ to each $w \in \RR^n$ as follows. First, if $w \in \RR^n$ is weakly increasing, there are unique indices $0 = a_0 < a_1 < \cdots < a_{m-1} < a_m = n$ with 
\begin{equation*}
    w_{a_0+1} = \cdots = w_{a_1} < w_{a_1+1} = \cdots = w_{a_2} < \cdots \cdots < w_{a_{m-1}+1} = \cdots = w_{a_m}
\end{equation*}
Define $\comp(w) = (a_1-a_0, a_2-a_1, \ldots, a_m-a_{m-1}) \vDash n$. For an arbitrary $w \in \RR^n$ we define $\comp(w)$ as $\comp(w')$ where $w'$ is the weakly increasing rearrangement of $w$.

\begin{example}
    $w = (-2, 0, 0, 0, 0.5, 1, 1)$ has composition $(1,3,1,2)$, as does any permutation of $w$.
\end{example}
 Thus, if $w$ is weakly increasing, then $w_i < w_{i+1}$ if and only if $i \in \set(\comp(w))$. For example, $\set(\comp(-2,0,0,0,0.5,1,1))$ is $\{1,4,5\}$.

An \emph{antichain} in a partially ordered set $(P,\leq)$ is a subset $A \subseteq P$ such that if $a_1, a_2 \in A$ are not equal, then $a_1 \not\leq a_2$ and $a_2 \not\leq a_1$.
\begin{lemma} \label{lem:antichain} If $v \in \RR^n$ is strictly increasing and $w \in \RR^n$ is weakly increasing, then $\{\sigma \in S_n : w \cdot \sigma v = 0\}$ is an antichain in the $\comp(w)$-Bruhat order. \end{lemma}

    \begin{proof}
        Set $\alpha = \comp(w)$. Suppose $\pi \in S_\alpha$ and $\tau = \word(B_\bullet)$ for some ordered set partition $B_\bullet$ of type $\alpha$. Let $i < j$ be such that $i$ appears in an earlier block of $B_\bullet$ than $j$ does, so that $\tau\pi <_{\alpha} t_{ij}\tau\pi$. Now,
        \begin{align} \label{eq:antichain}
            w\cdot \tau \pi v - w \cdot t_{ij} \tau \pi v &= w_i v_{\pi^{-1}\tau^{-1}(i)} - w_i v_{\pi^{-1}\tau^{-1}(j)} + w_j v_{\pi^{-1}\tau^{-1}(j)} - w_j v_{\pi^{-1}\tau^{-1}(i)} \nonumber \\
            &= (w_j - w_i)(v_{\pi^{-1}\tau^{-1}(j)} - v_{\pi^{-1}\tau^{-1}(i)}).
        \end{align}
        We claim that this quantity is strictly positive. First, the assumption  $i \in B_a$ and $j \in B_b$ with $a < b$ means that $w_i < w_j$, because $B_\bullet$ has type $\alpha = \comp(w)$. Also, if $\set(\alpha) = \{k_1 < \cdots < k_m\}$, then $\tau^{-1}(i) \in [k_{a-1}+1, k_a]$ and $\tau^{-1}(j) \in [k_{b-1}+1, k_b]$ by definition of $\tau = \word(B_\bullet)$. But $\pi^{-1}$ preserves both of those intervals because it is in $S_{\alpha}$, so $\pi^{-1}\tau^{-1}(i) < \pi^{-1}\tau^{-1}(j)$. The positivity of \eqref{eq:antichain} proves the lemma: if $\sigma <_{\alpha} \sigma'$, then $\sigma$ and $\sigma'$ are related by a sequence of relations of the type just considered, and so $w \cdot \sigma v - w \cdot \sigma' v > 0$.

    \end{proof}

The preceding lemma shows that we must bound the sizes of antichains in $\alpha$-Bruhat order, which can be done in terms of the following polynomials.
\begin{definition} The \emph{$q$-multinomial coefficient} associated to a composition $\alpha = (\alpha_1, \ldots, \alpha_m)$ of $n$ is the polynomial 
    \begin{equation*}
        \qbinom{n}{\alpha} := \frac{[n]_q!}{[\alpha_1]_q! \cdots [\alpha_m]_q!},
    \end{equation*}
    where $[n]_q = 1 + q + q^2 + \cdots + q^{n-1}$ and $[n]_q! = [1]_q [2]_q \cdots [n]_q$. As a special case, the \emph{$q$-binomial coefficient} $\smallqbinom{n}{k}$  is $\smallqbinom{n}{(k,n-k)}$.
\end{definition}
Let $\maxcoeff{f}$ denote the largest coefficient in a polynomial $f$. Recall that $\alpha!$ means $\prod_i \alpha_i!$ when $\alpha$ is a sequence.

\begin{thm} \label{thm:q-binomial-reduction}
    Let $v, w \in \RR^n$ where $v$ has distinct coordinates. Then
    \begin{equation*}
        \pair(v,w) := \#\{\sigma \in S_n : w \cdot \sigma v = 0\} \leq \comp(w)! \maxcoeff{\qbinom{n}{\comp(w)}}.
    \end{equation*}
\end{thm}

\begin{proof}
    Abbreviate $\comp(w)$ as $\alpha$.
Since $\pair(v,w) = \pair(\sigma v, \tau w)$ for any $\sigma, \tau \in S_n$, we can assume that $v$ is strictly increasing and $w$ is weakly increasing. By Lemma~\ref{lem:antichain}, $\{\sigma \in S_n : w \cdot \sigma v = 0\}$ is an antichain in  $\alpha$-Bruhat order, so there is also an antichain of size $\pair(v,w)$ in the isomorphic poset consisting of $\alpha!$ disjoint copies of $S_n/S_{\alpha}$.

Stanley \cite[Theorem 3.1]{sperner-hard-lefschetz} showed that the poset $S_n / S_\alpha$ has the \emph{Sperner property}: it is a ranked poset in which the set of elements of rank $r$, for some fixed $r$, form an antichain of maximal size. The rank generating function of $S_n / S_\alpha$ is $\smallqbinom{n}{\alpha}$ \cite[\S 4]{sperner-hard-lefschetz}, and so the Sperner property implies that the largest antichain in $S_n/S_\alpha$ has size $M(\smallqbinom{n}{\alpha})$. It is easy to see from this that the disjoint union of $\alpha!$ copies of $S_n / S_\alpha$ is also Sperner and has rank generating function $\alpha! \smallqbinom{n}{\alpha}$, from which the theorem follows.
\end{proof}

We now bound the coefficients of $\smallqbinom{n}{\alpha}$. Given two compositions $\alpha$ and $\beta$ of $n$, we say that $\beta$ \emph{refines} $\alpha$ and write $\alpha \prec \beta$ if one can obtain $\beta$ from $\alpha$ by a sequence of operations of the form 
\begin{equation*}
    \gamma \leadsto (\gamma_1, \ldots, \gamma_{i-1}, \gamma_i-p, p, \gamma_{i+1}, \ldots, \gamma_m) \quad \text{where $0 < p < \gamma_i$ and $1 \leq i \leq m$}.
\end{equation*} 
For example, $(1,4,3,1,2)$ refines $(1,7,1,2)$, as does $(1,2,2,3,1,2)$.
\begin{lemma}  \label{lem:refinement}
If $\alpha \prec \beta$ are compositions of $n$, then $\beta! \maxcoeff{\smallqbinom{n}{\beta}}\leq \alpha! \maxcoeff{\smallqbinom{n}{\alpha}}$. \end{lemma}

\begin{proof} 
    It suffices to assume that $\beta = (\alpha_1, \ldots, \alpha_{i-1}, \alpha_i-p, p, \alpha_{i+1}, \ldots, \alpha_m)$. Then
    \begin{equation*}
        \beta! \qbinom{n}{\beta} = {\alpha_i \choose p}^{-1} \qbinom{\alpha_i}{p} \cdot \alpha! \qbinom{n}{\alpha}.
    \end{equation*}
    As ${\alpha_i \choose p}^{-1} \smallqbinom{\alpha_i}{p}$ evaluated at $q = 1$ is $1$, the lemma follows from the simple fact that if $f$ and $g$ are polynomials with nonnegative coefficients, then $M(fg) \leq f(1)M(g)$.
\end{proof}

Lemma~\ref{lem:refinement} shows that $\max_{\alpha \vDash n}\alpha! \maxcoeff{\smallqbinom{n}{\alpha}}$ is achieved when $\alpha$ has length at most $2$. Writing $\smallqbinom{n}{k}$ for the $q$-binomial coefficient $\smallqbinom{n}{(k,n-k)}$, we must maximize $k!(n-k)!M(\smallqbinom{n}{k})$ for $0 < k < n$. The coefficient of $q^r$ in $\smallqbinom{n}{k}$ has a useful combinatorial interpretation: it is the number of partitions of $r$ into at most $k$ parts each of size at most $n-k$, i.e., integer sequences $n-k \geq \lambda_1 \geq \cdots \geq \lambda_k \geq 0$ with $\sum_i \lambda_i = r$.

From this interpretation one works out that $\maxcoeff{\smallqbinom{n}{2}} = 2\lfloor n/2 \rfloor (n-2)!$, and so Theorem~\ref{thm:main} is equivalent to the statement that $k!(n-k)!M(\smallqbinom{n}{k})$ is maximized when $k = 2$. The desired inequality $k!(n-k)!M(\smallqbinom{n}{k}) \leq 2\lfloor n/2 \rfloor (n-2)!$ is equivalent to $M(\smallqbinom{n}{k}) \leq \frac{1}{n} {n \choose k}$ for odd $n$, and to the weaker inequality $M(\smallqbinom{n}{k}) \leq \frac{1}{n-1} {n \choose k}$ for even $n$.

One could hope to prove $M(\smallqbinom{n}{k}) \leq \frac{1}{n} {n \choose k}$ combinatorially by explicitly dividing the set of all ${n \choose k}$ partitions with at most $k$ parts of size at most $n-k$ into $n$ disjoint groups of equal size in such a way that, for any fixed $r$, all the partitions of size $r$ are in one of the groups. An obvious candidate for such a grouping would be to divide up the partitions according to their sum modulo $n$, or equivalently, to consider $\smallqbinom{n}{k}$ modulo $q^n-1$. This does not quite work---the groups obtained this way need not have exactly equal sizes, and indeed $n$ may not even divide ${n \choose k}$---but it will be close enough to allow us to deduce the required inequality.

In the next lemma, $\mu$ and $\phi$ are the M\"obius function and totient function respectively, and $(a,b)$ means $\gcd(a,b)$.
\begin{lemma} \label{lem:congruence} Given $n \in \NN$ and $d \mid n$, define polynomials 
    \begin{equation*}
        F_{n,d}(q) = \begin{cases}
            \displaystyle \frac{1}{n}\sum_{i=0}^{n-1} \mu\left(\frac{d}{(d,i)}\right)\phi((d,i)) q^i & \text{if $d$ is squarefree}\\
            F_{n/e, d/e}(q^e) & \text{if $e^2 \mid d$ for some $e > 1$}            
        \end{cases}
    \end{equation*}
    Then
    \begin{equation*}
        \qbinom{n}{k} \equiv \sum_{d|(n,k)} {n/d \choose k/d} F_{n,d} \pmod{q^n-1}.
    \end{equation*}
\end{lemma}
As defined, the $F_{n,d}$ depend not only on $n$ and $d$ but on a choice of $e > 1$ with $e^2\mid d$ whenever $d$ is not squarefree. In fact they are independent of these choices, but as the choice of $e$ will not affect the proof, we omit the verification.

\begin{proof} Let $\Phi_d$ be the $d$\textsuperscript{th} cyclotomic polynomial, so that $q^n-1 = \prod_{d|n} \Phi_d(q)$. Sagan's formulas \cite[Theorem 2.2]{sagan-cyclotomic} for $\smallqbinom{n}{k}$ modulo $\Phi_d$ give
    \begin{equation*}
        \qbinom{n}{k} \equiv 
        \begin{cases}
            {n/d \choose k/d} & \text{if $d \mid (n,k)$}\\
            0 & \text{otherwise}
        \end{cases} \pmod{\Phi_d}
    \end{equation*}
 whenever $d \mid n$. By the Chinese remainder theorem, it therefore suffices to show that $F_{n,d}(q) \equiv \delta_{cd} \pmod{\Phi_c}$ for all $c \mid n$, or equivalently that $F_{n,d}(\zeta_c) = \delta_{cd}$, where $\zeta_c$ is a primitive $c$\textsuperscript{th} root of unity and $\delta$ is the Kronecker delta.

    First suppose that $e^2 \mid d$ for some $e > 1$. Given that $\zeta_c^e$ is a primitive $(c/(c,e))$\textsuperscript{th} root of unity, induction gives $F_{n,d}(\zeta_c) = F_{n/e,d/e}(\zeta_c^e) = \delta_{c/(c,e),d/e}$. If $\frac{d}{e} = \frac{c}{(c,e)}$, then $\frac{d}{e} \mid c$ and hence $e \mid c$ (because $e^2 \mid d$). This shows that $\delta_{c/(c,e),d/e} = \delta_{c/e,d/e} = \delta_{cd}$.

    By the previous paragraph, we are reduced to the case where $d$ is squarefree. For a fixed divisor $e\mid d$, we have
    \begin{equation*}
        \{i : 0 \leq i < n, (d,i) = e\} = \left\{je + rd : 0 \leq j < \frac{d}{e}, 0 \leq r < \frac{n}{d}, (j,d/e) = 1\right \},
    \end{equation*}
    so that 
    \begin{equation} \label{eq:F-divisor-formula}
        F_{n,d}(\zeta_c) = \frac{1}{n} \sum_{e\mid d} \mu\left(\frac{d}{e}\right)\phi(e)\sum_{\substack{0 \leq j < d/e\\ (j,d/e) = 1}}\sum_{0 \leq r < n/d} \zeta_c^{je + rd}.
    \end{equation}
    We consider the two inner sums here one at a time. First, $c\mid n$ implies $\tfrac{c}{(c,d)} \mid \tfrac{n}{d}$, so
    \begin{equation} \label{eq:sum-2}
        \sum_{0 \leq r < n/d} \zeta_c^{rd} = \sum_{0 \leq r < n/d} \zeta_{c/(c,d)}^r = \frac{n/d}{c/(c,d)} \sum_{0 \leq r < c/(c,d)} \zeta_{c/(c,d)}^r = \tfrac{n}{d}\delta_{c,(c,d)}.
    \end{equation}
    This means we may assume from now on that $c = (c,d)$, i.e. $c \mid d$. As for the first sum, observe that $\tfrac{c}{(c,e)} \mid \tfrac{d}{e}$, and that $\tfrac{d/e}{c/(c,e)}$ and $c/(c,e)$ are coprime because $d$ is squarefree. The Chinese remainder theorem therefore implies that
    \begin{gather*}
        \{0 \leq j < \tfrac{d}{e} : (j,\tfrac{d}{e}) = 1\} \to \{0 \leq j_1 < \tfrac{c}{(c,e)} : (j, \tfrac{c}{(c,e)}) = 1\} \times \{0 \leq j_2 < \tfrac{d/e}{c/(c,e)} : (j, \tfrac{d/e}{c/(c,e)}) = 1\},\\
        j \mapsto (j \bmod \tfrac{c}{(c,e)}, j \bmod \tfrac{d/e}{c/(c,e)})
    \end{gather*}
    is a bijection.  Since $\zeta_{c/(c,e)}^j$ depends only on $j \bmod \tfrac{c}{(c,e)}$, this bijection shows that
    \begin{equation} \label{eq:sum-1}
        \sum_{\substack{0 \leq j < d/e\\ (j,d/e) = 1}} \zeta_c^{je} = \sum_{\substack{0 \leq j < d/e\\ (j,d/e) = 1}} \zeta_{c/(c,e)}^j = \phi(\tfrac{d/e}{c/(c,e)}) \sum_{\substack{0 \leq j < c/(c,e)\\ (j,c/(c,e)) = 1}} \zeta_{c/(c,e)}^j = \phi(\tfrac{d/e}{c/(c,e)}) \mu(\tfrac{c}{(c,e)}),
    \end{equation}
    where we have used the general identity $\sum_{0 \leq j < m, (j,m) = 1} \zeta_m^j = \mu(m)$.

    Putting \eqref{eq:sum-2} and \eqref{eq:sum-1} into \eqref{eq:F-divisor-formula} and still assuming $c \mid d$, we get 
    \begin{equation*}
        F_{n,d}(\zeta_c) = \frac{1}{d} \sum_{e\mid d}\mu\left(\frac{d}{e}\right) \phi(e)\phi\left(\frac{d/e}{c/(c,e)}\right)\mu\left(\frac{c}{(c,e)}\right).
    \end{equation*}
    Using the multiplicativity of $\phi$ and $\mu$ and the coprimality of $e$ and $\tfrac{d/e}{c/(c,e)}$ and of $d/c$ and $(c,e)$, this is 
    \begin{equation*}
        F_{n,d}(\zeta_c) = \frac{1}{d} \phi\left(\frac{d}{c}\right) \mu\left(\frac{d}{c}\right) \sum_{e\mid d}\phi((c,e))\mu\left(\frac{e}{(c,e)}\right).
    \end{equation*}
    Because $d$ is squarefree, every divisor $e\mid d$ can be written uniquely as $e = fg$ where $f\mid c$ and $g \mid \tfrac{d}{c}$, and one then has $(c,e) = f$. Thus, 
    \begin{equation*}
        F_{n,d}(\zeta_c) = \frac{1}{d} \phi\left(\frac{d}{c}\right) \mu\left(\frac{d}{c}\right) \sum_{f\mid c}\sum_{g\mid d/c}\phi(f)\mu(g).
    \end{equation*}
    By M\"obius inversion, $\sum_{g \mid d/c} \mu(g) = \delta_{cd}$, and so 
    \begin{equation*}
        F_{n,d}(\zeta_c) = \delta_{cd} \frac{1}{d}  \sum_{f\mid d} \phi(f) = \delta_{cd},
    \end{equation*}
    where the last identity $\sum_{f\mid d}\phi(f) = d$ is due to Gauss.
\end{proof}
 
If $f(q)$ is a polynomial, we write $[q^r]f(q)$ for the coefficient of $q^r$ in $f(q)$.
\begin{lemma} \label{lem:q-binomial-inequality-1} Fix $n \geq 1$ and $1 < k < n-1$ and an integer $i$. Then
\begin{equation*}
\left| \frac{1}{n}{n \choose k} - \sum_{r \equiv i \bmod{n}}[q^r] \smallqbinom{n}{k} \right| \leq  \sqrt{\frac{1}{n} {n \choose k}}
\end{equation*} \end{lemma}

    \begin{proof} By Lemma~\ref{lem:congruence},
        \begin{equation*}
            \sum_{r \equiv i \bmod{n}} [q^r]\qbinom{n}{k} = [q^i] \sum_{d|(n,k)} {n/d \choose k/d} F_{n,d}.
        \end{equation*}
        Since $F_{n,1} = \frac{1}{n}[n]_q$ and the coefficients of $F_{n,d}$ are bounded by $\frac{d}{n}$ in absolute value,
        \begin{align} \label{eq:error}
           \left|  \sum_{r \equiv i \bmod{n}} [q^r]\qbinom{n}{k} - \frac{1}{n} {n \choose k}  \right| = \left|\sum_{\substack{d\mid (n,k)\\ d > 1}}  {n/d \choose k/d}[q^i] F_{n,d}\right| \leq  \sum_{\substack{d\mid (n,k)\\ d > 1}}  \frac{d}{n}{n/d \choose k/d}.
\end{align} 
The Vandermonde identity ${a_1 + \cdots + a_p \choose b} = \sum_{b_1 + \cdots + b_p = b} {a_1 \choose b_1} \cdots {a_p \choose b_p}$ implies ${n/d \choose k/d} \leq {n \choose k}^{1/d}$, so 
\begin{align} \label{eq:M-inequality}
    \left|  \sum_{r \equiv i \bmod{n}} [q^r]\qbinom{n}{k} - \frac{1}{n} {n \choose k}  \right| &\leq \sum_{\substack{d\mid (n,k)\\ d > 1}}  \frac{d}{n}{n \choose k}^{1/d} \leq  \sum_{d=2}^{k}  \frac{d}{n}{n \choose k}^{1/d} \nonumber \\
    & \leq \frac{9}{n} {n \choose k}^{1/2} + \sum_{d=5}^{k}  \frac{d}{n}{n \choose k}^{1/d}
 \end{align}
Let us assume that $n \geq 81$. If $k \leq 4$, then the lemma holds because $\frac{9}{n} \leq \frac{1}{\sqrt{n}}$, so assume $k \geq 5$ from now on. We claim that, for $5 \leq d \leq k$ and $n$ sufficiently large,
\begin{equation*}
    \frac{d}{n} {n \choose k}^{1/d} \leq \frac{1}{k-4} \left(\frac{1}{\sqrt{n}}{n \choose k}^{1/2} - \frac{9}{n}{n \choose k}^{1/2}\right).
\end{equation*}
Given \eqref{eq:M-inequality}, the lemma would follow from this claim. A little rearrangement shows that the claim is equivalent to the inequality $\left( \frac{d(k-4)}{\sqrt{n}-9} \right)^{\frac{2d}{d-2}} \leq {n \choose k}$. We may assume without loss of generality that $k \leq \frac{n}{2}$, in which case ${n \choose 5} \leq {n \choose k}$, so it suffices to show that $\left( \frac{d(n/2-4)}{\sqrt{n}-9} \right)^{\frac{2d}{d-2}} \leq {n \choose 5}$. It is a calculus exercise to show that $d^{\frac{2d}{d-2}}$ is decreasing on $[5,\gamma]$ and increasing on $[\gamma, \infty)$ where $\gamma \approx 5.36$, so $\max_{5 \leq d \leq n/2} d^{\frac{2d}{d-2}} = \max(5^{10/3}, (\frac{n}{2})^{\frac{2n}{n-4}})$. Assuming $n \geq 18^2$ so that $\sqrt{n}-9 \geq \frac{1}{2}\sqrt{n}$,
\begin{align*}
    \left( \frac{d(n/2-4)}{\sqrt{n}-9} \right)^{\frac{2d}{d-2}} &\leq \max(5^{10/3}, (\tfrac{n}{2})^{\frac{2n}{n-4}}) \max_{5 \leq d \leq k}\left(\frac{n/2}{\sqrt{n}/2}\right)^{\frac{2d}{d-2}} \leq \max(5^{10/3}, (\tfrac{n}{2})^{\frac{2n}{n-4}}) n^{5/3}\\
    &\leq \left(\frac{n}{2}\right)^{81/40} n^{5/3} \leq \frac{1}{4}n^4 \quad \text{(because we are assuming $n \geq 18^2$)}.
\end{align*}
Using bounds on roots one proves that $n^4/4 \leq {n \choose 5}$ when $n \geq 40$. It remains to verify the lemma for $n < 18^2$, which we have done by computer.

   \end{proof}

\begin{lemma} \label{lem:binomial-inequality} If $n \geq k \geq 2j \geq 2$, then ${n-j \choose k-j}^2 {n \choose k}^{-1} \geq (\frac{n}{k})^{\log (2)(k-2j)}$. \end{lemma}

    \begin{proof}
        Let $f(x) = (x-j)_{k-j} (x)_j^{-1}$, where $(x)_j$ is the falling factorial $x(x-1) \cdots (x-j+1)$. Then ${n-j \choose k-j}^2 {n \choose k}^{-1} = f(n)/f(k)$, and we compute
        \begin{equation*}
            \frac{f(x+1)}{f(x)} = \frac{(x-j+1)^2}{(x+1)(x-k+1)} \geq 1 + \frac{k-2j}{x} \quad \text{for $x > 0$ and $k \geq 2j$}.
        \end{equation*}
        Now using the inequality $1+x \geq 2^x$ for $0 \leq x \leq 1$,
        \begin{align*}
            \frac{f(n)}{f(k)} &= \prod_{x=k}^{n-1} \frac{f(x+1)}{f(x)} \geq \prod_{x=k}^{n-1} \left(1 + \frac{k-2j}{x}\right)\geq \prod_{x=k}^{n-1} \exp\left(\log(2) \frac{k-2j}{x}\right)\\
            & = \exp\left(\log(2)(k-2j) \sum_{x=k}^{n-1} \frac{1}{x}\right) \geq  \exp\left(\log(2)(k-2j) \log \frac{n}{k}\right) =  \left(\frac{n}{k}\right)^{\log (2)(k-2j)}.
        \end{align*}
    \end{proof}

\begin{lemma} \label{lem:q-binomial-inequality-2} For $n \geq 0$ and $2 < k < n-2$, it holds that $\maxcoeff{\smallqbinom{n}{k}} \leq \frac{1}{n}{n \choose k}$. \end{lemma}

\begin{proof} Let $\smallqbinom{n}{k} = \sum_{j} a_j q^j$ and set $m_{n,k} = \lfloor \frac{k(n-k)}{2} \rfloor$. The polynomial $\smallqbinom{n}{k}$ has degree $k(n-k)$, and Sylvester \cite{sylvester-unimodality} proved that it is \emph{unimodal} in the sense that $1 = a_0 \leq a_1 \leq \cdots \leq a_{m_{n,k}} \geq a_{m_{n,k}+1} \geq \cdots \geq a_{k(n-k)} = 1$, so that $\maxcoeff{\smallqbinom{n}{k}} = a_{m_{n,k}}$. It is also a palindromic polynomial: $a_i = a_{k(n-k)-i}$ for any $i$.

    Lemma~\ref{lem:q-binomial-inequality-1} shows that $\sum_{r \equiv m_{n,k} \bmod{n}} a_{r} \leq \tfrac{1}{n}{n \choose k} + \sqrt{\tfrac{1}{n} {n\choose k}}$, so it would suffice to find $r \equiv m_{n,k} \bmod{n}$ but $r \neq m_{n,k}$ such that $a_r$ (or a sum of such $a_r$) exceeds $\sqrt{\tfrac{1}{n}{n \choose k}}$. We do this in two ways: one which works for all but finitely many $k$, and then another which works for any fixed $k$ and all but finitely many $n$.
    \begin{enumerate}[(a)]
        \item Assume $n \geq 16$. Then $2(n-8) \geq n$, so any interval $[a,a+n)$ can contain at most two multiples of $n-8$. Unimodality then implies the first inequality in
    \begin{equation*}
        \sum_{i=2}^{\infty} [q^{m_{k,n}-i(n-8)}]\qbinom{n}{k} \leq \sum_{i=1}^{\infty} 2[q^{m_{k,n}-in}]\qbinom{n}{k} \leq \sum_{\substack{r \equiv m_{n,k} \bmod{n} \\ r \neq m_{n,k}}} [q^r]\qbinom{n}{k},
    \end{equation*}
    while the second follows from $a_{m_{n,k}-i} \leq a_{m_{n,k}+i}$ for $i \geq 0$, a consequence of unimodality and palindromicity.

    The partition interpretation of $\smallqbinom{n}{k}$ makes clear that $\smallqbinom{n}{k} - \smallqbinom{n-1}{k-1}$ has nonnegative coefficients, so 
    \begin{equation*}
        \sum_{i=2}^{\infty} [q^{m_{k,n}-i(n-8)}]\qbinom{n}{k} \geq  \sum_{i=2}^{\infty} [q^{m_{k,n}-i(n-8)}]\qbinom{n-8}{k-8}.
    \end{equation*}
    We can assume without loss of generality that $k \leq \tfrac{n}{2}$, in which case $m_{k,n} - 2(n-8) > m_{k-8,n-8}$ (this explains the appearance of $8$: the inequality would be false with $8$ replaced by $7$). The integers $m_{k,n} + j(n-8)$ for varying $j$ therefore look like
    \begin{center}
        \begin{tikzpicture}
            \draw  (-5,0) -- (8,0);
            \draw (0,-.2) -- (0,.2);
            \node[above] at (0,.2) {$\scriptstyle (k-8)(n-k)/2$};
            \node[below] at (8,-0.1) {$\cdots$};
            \filldraw (7,0) circle (2pt) node[below] {$2$};
            \node[below] at (6,-0.1) {$\cdots$};
            \filldraw (4.5,0) circle (2pt) node[below] {$p-2$};
            \filldraw (2.5,0) circle (2pt) node[below] {$p-1$};
            \filldraw (0.5,0) circle (2pt) node[below] {$p$};
            \filldraw (-1.5,0) circle (2pt) node[below] {$p+1$};
            \filldraw (-3.5,0) circle (2pt) node[below] {$p+2$};
            \node[below] at (-4.5,-0.1) {$\cdots$};
        \end{tikzpicture}
    \end{center}
    or like 
    \begin{center}
        \begin{tikzpicture}
            \draw  (-5,0) -- (8,0);
            \draw (-1,-.2) -- (-1,.2);
            \node[above] at (-1,.2) {$\scriptstyle (k-8)(n-k)/2$};
            \node[below] at (8,-0.1) {$\cdots$};
            \filldraw (7,0) circle (2pt) node[below] {$2$};
            \node[below] at (6,-0.1) {$\cdots$};
            \filldraw (4.5,0) circle (2pt) node[below] {$p-2$};
            \filldraw (2.5,0) circle (2pt) node[below] {$p-1$};
            \filldraw (0.5,0) circle (2pt) node[below] {$p$};
            \filldraw (-1.5,0) circle (2pt) node[below] {$p+1$};
            \filldraw (-3.5,0) circle (2pt) node[below] {$p+2$};
            \node[below] at (-4.5,-0.1) {$\cdots$};
        \end{tikzpicture}
    \end{center}
    where $j$ is indicated below the number line, and the two cases are separated based on which of $m_{k,n} - p(n-8)$ and $m_{k,n} - (p+1)(n-8)$ is closer to $\frac{(k-8)(n-k)}{2}$, the center of the list of coefficients of $\smallqbinom{n-8}{k-8}$. Pair up the integers as $(p+1, p-1), (p+2, p-2), \ldots$ in case 1 (leaving $p$ unpaired) or $(p+1,p), (p+2,p-1), \ldots$ in case 2. Then each pair $(a,b)$ has the property that $a$ is weakly closer to $\frac{(k-8)(n-k)}{2}$ than $b$ is, and so unimodicity and palindromicity says $[q^{m_{k,n}-a(n-8)}]\smallqbinom{n-8}{k-8} \geq [q^{m_{k,n}-b(n-8)}]\smallqbinom{n-8}{k-8}$. Because $p \geq 2$, this implies
    \begin{equation*}
        \sum_{i=2}^{\infty} [q^{m_{k,n}-i(n-8)}]\qbinom{n-8}{k-8} \geq \frac{1}{2}\sum_{r \equiv m_{k,n} \bmod{n-8}} [q^r]\qbinom{n-8}{k-8}.
    \end{equation*}

    Applying Lemma~\ref{lem:q-binomial-inequality-1} to this last expression and concatenating the inequalities in the last three displayed equations,
    \begin{equation*}
        \sum_{\substack{r \equiv m_{n,k} \bmod{n} \\ r \neq m_{n,k}}} [q^r]\qbinom{n}{k} \geq \frac{1}{2}\left(\frac{1}{n-8}{n-8 \choose k-8} - \sqrt{\frac{1}{n-8}{n-8 \choose k-8}}\right).
    \end{equation*}

    We must therefore show that $\sqrt{\frac{1}{n}{n\choose k}} \leq \frac{1}{2}\left(\frac{1}{n-8}{n-8 \choose k-8} - \sqrt{\frac{1}{n-8}{n-8 \choose k-8}}\right)$. As 
    \begin{equation*}
        \sqrt{\frac{1}{n}{n\choose k}} + \frac{1}{2}\sqrt{\frac{1}{n-8}{n-8\choose k-8}} \leq \frac{3}{2}\sqrt{\frac{1}{n-8}{n \choose k}},
    \end{equation*}
    it would suffice to show $\frac{3}{2}\sqrt{\frac{1}{n-8}{n \choose k}} \leq \frac{1}{2}\frac{1}{n-8}{n-8 \choose k-8}$, i.e. $9(n-8){n \choose k} \leq {n-8 \choose k-8}^2$.

    By Lemma~\ref{lem:binomial-inequality}, if $k \geq 16$ then
    \begin{equation*}
        {n-8 \choose k-8}^2 {n \choose k}^{-1} \geq \left(\frac{n}{k}\right)^{\log(2)(k-16)}.
    \end{equation*}
    As a function of $k$, $\left(\frac{n}{k}\right)^{\log(2)(k-16)}$ is increasing near $0$ and decreasing near $\infty$ with exactly one critical point in between. Thus, its minimum over any interval $a \leq k \leq b$ occurs at either $k = a$ or $k = b$. In particular, if we assume $k \geq 20$, then 
    \begin{equation*}
        {n-8 \choose k-8}^2 {n \choose k}^{-1} \geq \min\left( \left(\frac{n}{20}\right)^{4\log 2}, 2^{\log(2)(n/2-16)} \right).
    \end{equation*}
    The right side exceeds $9(n-8)$ for $n \geq 370$.

    \item Suppose $k \geq 5$ and $n \geq 16$. Then $m_{n,k} - (n-k) \geq n$, so the interval $[n-k, m_{n,k})$ contains some number $h$ with $h \equiv m_{n,k} \bmod{n}$. Unimodicity implies $[q^h]\smallqbinom{n}{k} \geq [q^{n-k}]{n \choose k}$, so it suffices to show that $[q^{n-k}]{n \choose k} \geq \sqrt{\frac{1}{n} {n \choose k}}$. Observe that $[q^{n-k}]\smallqbinom{n}{k}$ is simply the number of partitions of $n-k$ into $k$ parts (some of which may be $0$). Any sequence in $\{0,1,2,\ldots\}^k$ with sum $n-k$ may be obtained by permuting the parts of such a partition, and there are ${n-1 \choose k-1}$ such sequences, so $[q^{n-k}]\smallqbinom{n}{k} \geq \frac{1}{k!} {n-1 \choose k-1}$.
    
    A little rearrangement shows that $\frac{1}{k!} {n-1 \choose k-1} \geq \sqrt{\frac{1}{n} {n \choose k}}$ if and only if $(n-1)_{k-1} \geq k(k-1)!^3$. For any fixed $k$, one can simply compute the minimal $n_0 \geq k$ with $(n_0-1)_{k-1} \geq k(k-1)!^3$, and then the inequality $(n-1)_{k-1} \geq k(k-1)!^3$ will hold for all $n \geq n_0$, because $(n-1)_{k-1}$ is an increasing function of $n$ for $n \geq k$. Carrying out this computation, one concludes that the desired inequality holds for $5 \leq k \leq 20$ and $n \geq 594$.

    \item The missing cases are when $n < 594$ or $k \in \{3,4\}$. One has $m_{n,4} = 2(n-4) \equiv n-8 \pmod{n}$, so if $n \geq 8$ and $k = 4$ then one can use the approach of (b), replacing $n-k$ with $n-8$. Similarly, $m_{n,3} \equiv \lfloor \frac{n-9}{2} \rfloor \pmod{n}$. In this way one deduces the lemma when $k = 3$ and $n \geq 31$, or $k = 4$ and $n \geq 19$. We have checked the lemma for $n < 594$ by computer: some work can be saved here by observing that \eqref{eq:error} in the proof of Lemma~\ref{lem:q-binomial-inequality-1} implies the lemma immediately when $\gcd(n,k) = 1$, eliminating about 60\% of the pairs that would otherwise have to be checked.
\end{enumerate}
    
\end{proof}

We now restate and prove Theorem~\ref{thm:main}. Recall that $\one = (1,1,\ldots,1) \in \RR^n$.

\begin{thm*}[Theorem \ref{thm:main}] For $n \geq 3$, the maximum of $\pair(v,w) = \#\{\sigma \in S_n : w \cdot \sigma v = 0\}$ over all $v,w \in \RR^n$ where $v$ has distinct coordinates and $v \cdot \one \neq 0$ is
    \begin{equation*}
        2\lfloor \tfrac{n}{2} \rfloor (n-2)! = 
        \begin{cases}
            (n-1)! & \text{for $n$ odd}\\
            n(n-2)! & \text{for $n$ even}
        \end{cases}.
    \end{equation*}
\end{thm*}

\begin{proof}
   The only vectors $w$ with $\comp(w) = (n)$ are those in $\RR \one$, so if $\comp(w) = (n)$ then $v \cdot \one \neq 0$ implies $\pair(v,w) = 0$. We can therefore assume $\comp(w)$ has at least two parts. Now,
   \begin{align*}
    \pair(v,w) &\leq \comp(w)! \maxcoeff{\qbinom{n}{\comp(w)}}  \quad \text{(by Theorem~\ref{thm:q-binomial-reduction})}\\
    &\leq \max_{\substack{\alpha \vDash n \\ \ell(\alpha) \geq 2}} \alpha! \maxcoeff{\qbinom{n}{\alpha}}\\
    &= \max_{0 < k < n} k!(n-k)! \maxcoeff{\qbinom{n}{k}}   \quad \text{(by Lemma~\ref{lem:refinement})}.
   \end{align*}By Lemma~\ref{lem:q-binomial-inequality-2}, if $2 < k < n-2$ then
   \begin{equation*}
    k!(n-k)! \maxcoeff{\qbinom{n}{k}} \leq k!(n-k)! \frac{1}{n} {n \choose k} = (n-1)!. 
   \end{equation*}
   Using the interpretation of $[q^r] \smallqbinom{n}{k}$ as the number of partitions of $r$ into at most $k$ parts of size at most $n-k$, we see that 
   \begin{align*}
    &1!(n-1)!\maxcoeff{\qbinom{n}{1}} = (n-1)!\\
    &2!(n-2)!\maxcoeff{\qbinom{n}{2}} = 2(n-2)!\lfloor \tfrac{n}{2} \rfloor = \begin{cases}
        (n-1)! & \text{for $n$ odd}\\
        n(n-2)! & \text{for $n$ even}.
    \end{cases}
\end{align*}
\cite[Examples 1.3 and 1.4]{Sn-orbit-hyperplane} provide specific vectors $v$ and $w$ for which the upper bound in the theorem is attained.
\end{proof}

\subsection{Final remarks}
A sequence $a_1, a_2, \ldots, a_n$ is called \emph{log-concave at $k$} if $a_k^2 \geq a_{k-1} a_{k+1}$, and just \emph{log-concave} if it is log-concave at every $1 < k < n$. For instance, the binomial coefficients $a_k = {n \choose k}$ for a fixed $n$ are log-concave. At one point in the development of this paper we thought to use the log-concavity of the sequence of coefficients of each fixed $\smallqbinom{n}{k}$---but as pointed out in \cite{stanley-log-concave}, this property does not actually hold. However, computations up to $n = 430$ suggest the following conjecture.

\begin{conjecture*}
    For any fixed $n \geq 45$ and $13 \leq k \leq n-13$, the sequence of coefficients $[q^r]\smallqbinom{n}{k}$ is log-concave at each $25 < r < k(n-k)-25$.
\end{conjecture*}
DeSalvo and Pak \cite{desalvo-pak} have shown that the sequence of partition numbers $p(r)$ counting all partitions of $r$ is log-concave at all $r > 25$.


\bibliographystyle{plain}
\bibliography{../../bib/algcomb}

\end{document}

%% file: preamble.tex
\newcommand{\RR}{\mathbb{R}}
\newcommand{\NN}{\mathbb{N}}

\newcommand{\pair}{\mathcal{O}}
\newcommand{\comp}{\operatorname{comp}}
\newcommand{\smallqbinom}[2]{\left[\begin{smallmatrix} #1 \\ #2 \end{smallmatrix}\right]}
\newcommand{\qbinom}[2]{\begin{bmatrix} #1 \\ #2 \end{bmatrix}}
\newcommand{\maxcoeff}[1]{M\!\left(#1\right)}

\newcommand{\set}{S}
\newcommand{\one}{\mathbf{1}}
\DeclareMathOperator{\word}{word}